\documentclass[11pt,english]{amsart}

\newcommand{\balpha}{\boldsymbol\alpha}
\newcommand{\bbeta}{\boldsymbol\beta}
\newcommand{\bz}{\boldsymbol z}
\newcommand{\bw}{\boldsymbol w}
\newcommand{\bx}{\boldsymbol{x}}
\newcommand{\by}{\boldsymbol{y}}
\newcommand{\bxs}{\boldsymbol{x_s}}
\newcommand{\bys}{\boldsymbol{y_s}}
\usepackage[T1]{fontenc}
\usepackage{float}

\usepackage[english]{babel}
\usepackage{amssymb,url,xspace}
\usepackage{graphicx}
\usepackage{subfig}
\usepackage{xcolor}
\usepackage{textcmds}
\usepackage[section]{placeins}
\usepackage{hyperref}

\usepackage{amsmath}
\usepackage{amssymb}
\usepackage{amsfonts}
\usepackage{amsthm}

\theoremstyle{plain}
\newtheorem{theo}{Theorem}
\newtheorem{coro}[theo]{Corollary}

\newtheorem{lemm}[theo]{Lemma}

\theoremstyle{definition}

\theoremstyle{remark}

\numberwithin{equation}{section}
\numberwithin{theo}{section}

\title[]{Twisting and stabilization in Knot Floer homology}
\date {December, 2023}

\author{Soheil Azarpendar}
\address{Mathematical Institute, University of Oxford, Andrew Wiles Building, Radcliffe
Observatory Quarter, Woodstock Road, Oxford, OX2 6GG, UK}
\email{azarpendar@maths.ox.ac.uk} 

\begin{document}

\begin{abstract}
Let $K_m$ be the result of applying $m$ full twists to $n$ parallel strands in a knot $K$. We prove that extremal knot Floer homologies of $K_m$ stabilize as $m$ goes to infinity. 
\end{abstract}

\maketitle

\section{Introduction}
Assume you have a diagram $D$ of the knot $K$, and some part of the diagram consists of $n$-parallel strands. More precisely, this means there is a disk $B$ in the plane such that the intersection of $D$ with $B$ looks like the identity $n$-braid (preserving the orientations) i.e. $$B \cap D \simeq I \times \{1,\dots,n\}.$$ 
We can define a family of knots $K_m$ by applying $m$ full twists to these strands. More precisely, a diagram $D_m$ for $K_m$ can be constructed by replacing $B \cap D$ with the diagram of the braid word $$((\sigma_1\sigma_2\dots\sigma_{n-1})^n)^m,$$
which is $m$ full twists on $n$ strands. You can see an example of twisting on two strands in Figure \ref{2strand3twist}. Figure \ref{3strandfulltwist} shows a full twist on 3 strands. Note that applying full twists doesn't change the number of components of the knot and preserves its orientation.\\

\begin{figure}[h]
\centering
\begin{center}
\includegraphics[scale=0.5]{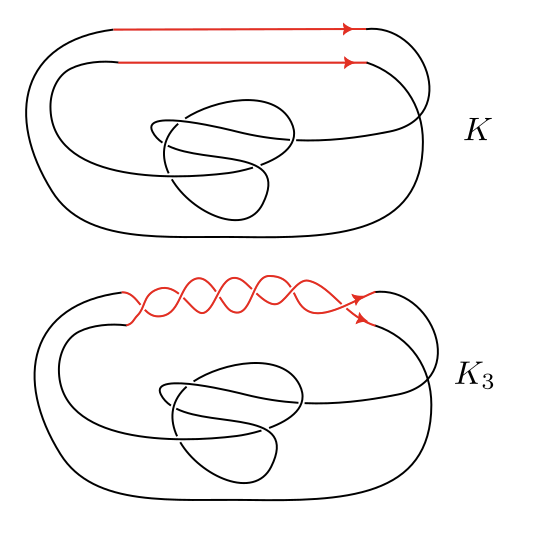}
\end{center}
\caption{Three full twists on 2 strands}\label{2strand3twist}
\end{figure}
\begin{figure}[h]
\centering
\begin{center}
\includegraphics[scale=0.4]{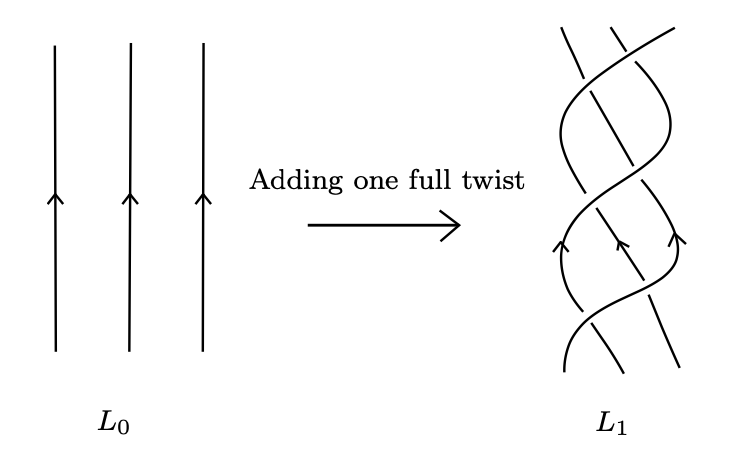}
\end{center}
\caption{A full twist on 3 strands\cite{chen2022twistings}}\label{3strandfulltwist}
\end{figure}
We want to investigate stabilization phenomena in $\widehat{HFK}(K_m)$ as $m \rightarrow \infty$. These stabilizations arise as categorifications of results about the Alexander polynomial. The most well-behaved case is twisting 2 parallel strands in a knot diagram. In this case, the total coefficient sequence of $\Delta(K_m)$ stabilises for $m>>0$. This is formulated in Theorem \ref{Alexpolystablize} taken from a paper by Lambert-Cole\cite{LambertCole2016TwistingMA}. 
\begin{theo}\label{Alexpolystablize}
Let $K_m$ be the result of applying $m$ full twists on two parallel strands in $K$. 
    There exists some $k>0$, some $d \in \mathbb{Z}$, and some polynomial $f(t)\in \mathbb{Z}[t]$ such that for $m$ sufficiently large, the Alexander polynomial of $K_m$ has the form 
    $$\Delta_{K_m}(t)=t^{m-k} f(t) + d \  \Delta_{T(2,m-k)}(t)+t^{-(m-k)} f(t).$$
\end{theo}

Lambert-Cole proves a categorification of Theorem \ref{Alexpolystablize} in the same paper\cite{LambertCole2016TwistingMA}. This stabilization of knot Floer homology is stated in Theorem \ref{L-Cstablize}. 
\begin{theo}\label{L-Cstablize}
Let $K_m$ be the result of applying $m$ full twists on two parallel strands in $K$. There exist some $k>0$ such that for $|m|$ sufficiently large, the knot Floer homology of $K_m$ satisfies 
\begin{equation*}
 \begin{aligned}
\widehat{HFK}(K_m,j)\cong\widehat{HFK}(K_{m+1},j+1) \ \ \text{for} \ \ j\geq -k, \\ 
\widehat{HFK}(K_m,j)\cong\widehat{HFK}(K_{m+1},j-1)[2] \ \ \text{for} \ \ j\leq k,
\end{aligned}
\end{equation*}
where $[i]$ denotes the shift in Maslov grading. 
\end{theo}
Note that the second statement follows from combining the first statement with symmetry properties of knot Floer homology.\\

Also note that there are two stabilization effects, one with a positive shift in the Alexander grading and one with a negative, and each works for a specific interval in the range of Alexander gradings. These two intervals overlap in the case of twisting two strands, which in turn means that similar to the Alexander polynomial, total shape of the knot Floer homology stabilizes.\\

When the number of strands is greater than two, we can’t have a result with this strength. However, there will be a stabilization phenomenon in the extremal Alexander gradings. This phenomenon for the Alexander polynomial was formulated by Daren Chen  \cite{chen2022twistings}. 
\begin{theo}\label{Chenstablize}
    Let $L_m$ be the result of applying $m$ full twists on $n$ parallel strands in a link $L$. There exists a Laurent series $h_{L}(t)$ with finitely many terms of negative degree in $t$, and some integer $r \in [\frac{n-1}{2},n-1]$, such that for any $k \in \mathbb{N}$, there exists $N \in \mathbb{N}$ where for any $m \geq N$, the first $k$ terms in the increasing order of degree of $t$ of $\Delta_{L_m}$ agree with the first $k$ term of 
    $$t^{\frac{mn(n-1-2r)}{2}} \ h_{L}(t).$$
\end{theo}
Note that Chen proves Theorem 1.3 for links, and in this case, even the growth of the degree of Alexander depends on the structure of the link. In this preprint, we focus on knots.\\

In this preprint, we modify Lambert-Cole's methods to derive a categorification of Chen's result. This result is formulated in the following Theorem. 

\begin{theo}\label{MainTheorem}
Let $K_m$ be the result of applying $m$ full twists on $n$ parallel strands in a link $K$. There exists a series of graded vector spaces $\{A_{-i}\}_{i \in \mathbb{N}}$ such that, for any $k \in \mathbb{N}$, there exists $N \in \mathbb{N}$ where for any $m \geq N$, the first $k$ non-trivial knot Floer homologies $\widehat{HFK}(K_m,j)$ in the decreasing order of Alexander grading agrees with the first $k$ vector spaces in the sequence $\{A_i\}$ up to a shift in the grading (denoted by $f_{m}$) i.e. $$\widehat{HFK}(K_m,g(K_m)-j)\cong A_{-j}\ [f_{m}] \ \ \text{for} \ \ 0\leq j \leq k-1.$$
\end{theo}
Note that the Theorem doesn't talk about the stabilization phenomena on absolute Maslov gradings. This is subject of ongoing research.\\

We decided to change some of the minor details of Theorem \ref{Chenstablize} in the phrasing of Theorem \ref{MainTheorem} so it becomes closer to our method of proof. This difference is superficial.\\

Also, note that another evidence for Theorem \ref{MainTheorem} comes from Hedden's formula for knot Floer homology of cables \cite{Hedden}.\\

This preprint is organized as follows. In Section \ref{Section2} we briefly recall the setting of knot Floer homology. Our setting is identical to Lambert-Cole's\cite{LambertCole2016TwistingMA}. In Section \ref{MainTheorem}, we discuss the proof of Theorem \ref{MainTheorem}.  
\section*{Acknowledgement}
It is a pleasure to thank my advisor, Professor András Juhász, as without his patience and guidance, this project wouldn't be possible. I must deeply thank Daren Chen for proposing the topic of this project, and for our helpful discussions. I must also thank Professor Cladius Zibrowius for his insightful suggestions.  
\section{Brief recollection of knot Floer homology}\label{Section2}
A multi-pointed Heegaard diagram $\mathcal{H}=(\Sigma, \balpha, \bbeta,\bz,\bw)$ for a knot $K \subset S^3$ is a tuple consisting of a genus $g$ Riemann surface $\Sigma$, two multicurves $\balpha=\{\alpha_1,\dots,\alpha_{g+n}\}$ and $\bbeta=\{\beta_1,\dots,\beta_{g+n}\}$, and two collections of basepoints $\bz=\{z_1,...,z_{n-1}\}$ and $\bw=\{w_1,\dots,w_{n-1}\}$ such that: 
\begin{itemize}
  \item $(\Sigma,\balpha,\bbeta)$ is a Heegaard diagram for $S^3$,
  \item each component of $\Sigma \setminus \balpha$ and $\Sigma \setminus \bbeta$ contains exactly one $\bz-$basepoint and one $\bw-$basepoint,
  \item the basepoints $\bz,\bw$ determine the link $L$ as follows: choose collections of embedded arcs $\gamma_1,\dots \gamma_{n-1}$ in $\Sigma \setminus \balpha$ and $\{\delta_1,\dots,\delta_{n-1}\}$ in $\Sigma \setminus \bbeta$ connecting the basepoints. Then after pressing the arcs $\{\gamma_i\}$ into the $\alpha-$handlebody and the arcs $\{\delta_j\}$ into $\beta-$handlebody, their union is $K$ ($K$ is in the bridge position with respect to $\Sigma$).
\end{itemize}
From a multipointed Heegaard diagram $\mathcal{H}$, we obtain a complex $\widehat{CFK}(\mathcal{H})$. In the symmetric product $Sym^{g+n}(\Sigma)$ of the Heegaard surface, the multicurves $\balpha,\bbeta$ determine $(g+n)-$dimensional tori $$\mathbb{T}_{\alpha}\cong\alpha_1\times\dots\times\alpha_{g+n} \ \text{and} \ \mathbb{T}_{\beta}\cong\beta_1\times\dots\times\beta_{g+n}.$$ 
The knot Floer complex $\widehat{CFK}(\mathcal{H})$ is freely generated over $\mathbb{F}[U_{w_2},\dots,U_{w_n}]$ by the intersection points $\mathfrak{G}(\mathcal{H}):=\mathbb{T}_{\alpha} \cap \mathbb{T}_{\beta}$ which are the $g-$tuples ${\boldsymbol{x}=(x_1,\dots,x_g)}$ such that $x_i \in \alpha_i \cap \beta_{\sigma(i)}$ for some permutation $\sigma \in S_g$. The complex possesses two gradings, the Maslov (homological) grading and the Alexander grading. Given any Whitney disk $\phi \in \pi_2(\bx,\by)$ connecting the two generators, the relative gradings of two generators $\bx,\by\in \mathbb{T}_{\alpha} \cap \mathbb{T}_{\beta}$ satisfy the formulas 
\begin{equation}
    M(\bx)-M(\by) = \mu(\phi) - 2n_{\bw}(\phi)
\end{equation}
\begin{equation}
    A(\bx)-A(\by)=n_{\bz}(\phi)-n_{\bw}(\phi).
\end{equation}
Multiplication by any formal variable $U_{w_i}$ drops the Maslov grading by $2$ and the Alexander grading by $1$. The subspace of elements with Maslov grading $i$ and Alexander grading $j$ is denoted by $\widehat{CFK}_i(\mathcal{H},j)$. The Alexander grading is pinned down by the assumption that the Euler characteristic of $\widehat{CFK}_{*}(\mathcal{H},j)$ is equal to the coefficient of $t^j$ in $\Delta_K(t)$, the symmetrized Alexander polynomial of K.\\

The differential is defined as follows: 
$$\widehat{\partial}\bx := \sum_{y \in \mathbb{T}_{\alpha} \cap \mathbb{T}_{\beta}} \sum_{\substack{\phi \in \pi_2(\bx,\by), \\ \mu(\phi)=1, \\ n_{\bz}(\phi)=0,\\ n_{w_1}(\phi)=0}} \# \widehat{\mathcal{M}}_{J_s}(\phi) \ U_{w_2}^{n_{w_2}(\phi)} \dots U_{w_{n-1}}^{n_{w_{n-1}}(\phi)}\ \by$$
where $\widehat{\mathcal{M}}_{J_s}(\phi)$ is the reduced moduli space of holomorphic representatives of $\phi$ w.r.t to a generic path of complex structures $Sym^{g+n}(\Sigma)$ denoted by $J_s$. Note that this differential drops the Maslov grading by one and preserves the Alexander grading.\\

The hat version of the knot Floer homology $\widehat{HFK}_{*}(K)$ is defined as the homology of the complex $(\widehat{CFK}(\mathcal{H}),\widehat{\partial})$.\\

\section{Proof of Theorem \ref{MainTheorem}}\label{ProofSection}
Similar to Lambert-Cole, we start with finding a bridge presentation for the knot $K_m$ from a bridge presentation of $K$ consisting of $T+1$ bridges $a_0 , \dots , a_T$ and $T+1$ overstrands $b_0,\dots,b_T$. Assume that $a_i$ is oriented from $z_i$ to $w_i$, and $b_i \cap a_i=z_i$. We can always take a bridge presentation for $K$ such that the chosen $n$ parallel strands also take the shape of $n$ parallel strands containing $z_0,\dots,z_{n-1}$ respectively. The bridge diagram for 3-strands looks like Figure \ref{3strandbridge} in a local neighborhood.\\

\begin{figure}[h]
\centering
\begin{center}
\includegraphics[scale=0.4]{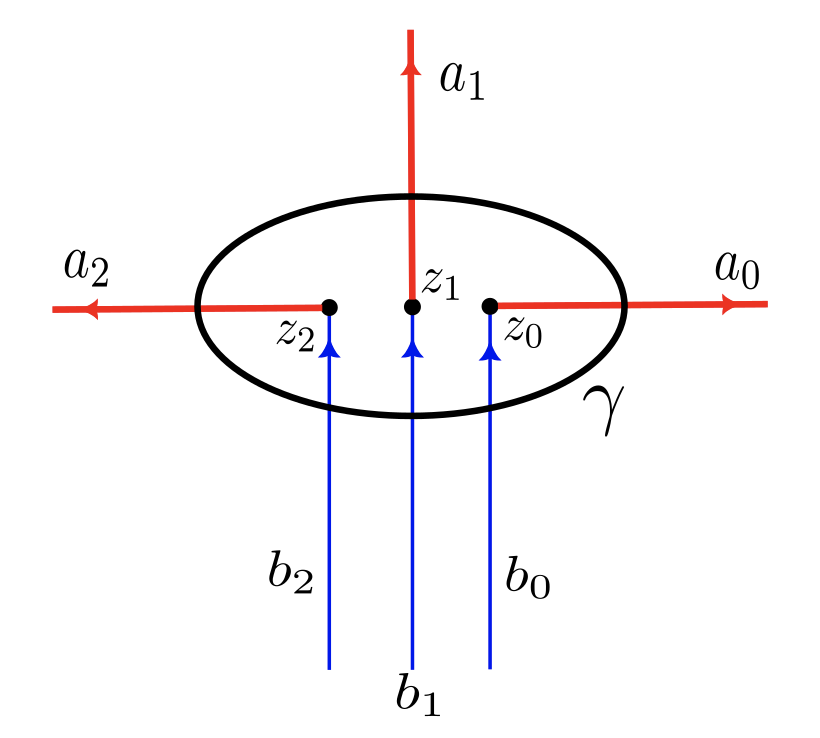}
\end{center}
\caption{3 parallel strands in the bridge position}\label{3strandbridge}
\end{figure}
Similar to Figure \ref{3strandbridge}, assume that $\gamma$ is a curve containing $z_0,\dots,z_{n-1}$. Orient $\gamma$ counter-clockwise as the boundary of the disk containing $z_0,\dots,z_{n-1}$. One can build a bridge diagram for $K_m$ by applying $m$ negative Dehn twists along $\gamma$ to $b_0 , \dots b_{n-1}$. This can be seen in Figure \ref{3strandtwistbridge}. One way to see this is by simplifying the diagram in case of one Dehn twist and using induction. The simplification is done for three strands in Figure \ref{simplifytwisted}, and can be generalized to more strands.\\

A more general argument comes from the surgery description of full twist operation. It is a well-known fact that knot $K_m$ can be obtained from $K$ by $-\frac{1}{m}$ surgery on an unknot that links the $m$ strands positively and geometrically $m$ times (each strand once). An example of such an unknot is the curve $\gamma$. Dehn twist on a curve on the boundary of a Heegaard handlebody corresponds to composition of the gluing map with the Dehn twist. In this case, applying $m$ Dehn twist around $\gamma$ on the boundary of $\beta-$handlebody is equivalent to $-\frac{1}{m}$ surgery on the curve $\gamma$ in $S^3$ (See Figures \ref{surgerytwist} and \ref{Dehntwistsurgery}).\\
\begin{figure}[h]
\centering
\begin{center}
\includegraphics[scale=0.5]{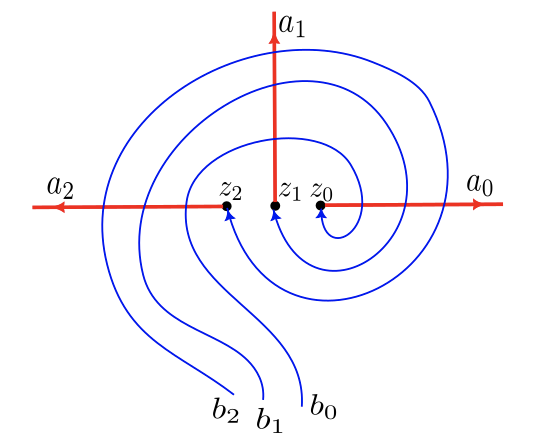}
\end{center}
\caption{One full twist on 3 strands in bridge position}\label{3strandtwistbridge}
\end{figure}
\begin{figure}[h]
\centering
\begin{center}
\includegraphics[scale=0.4]{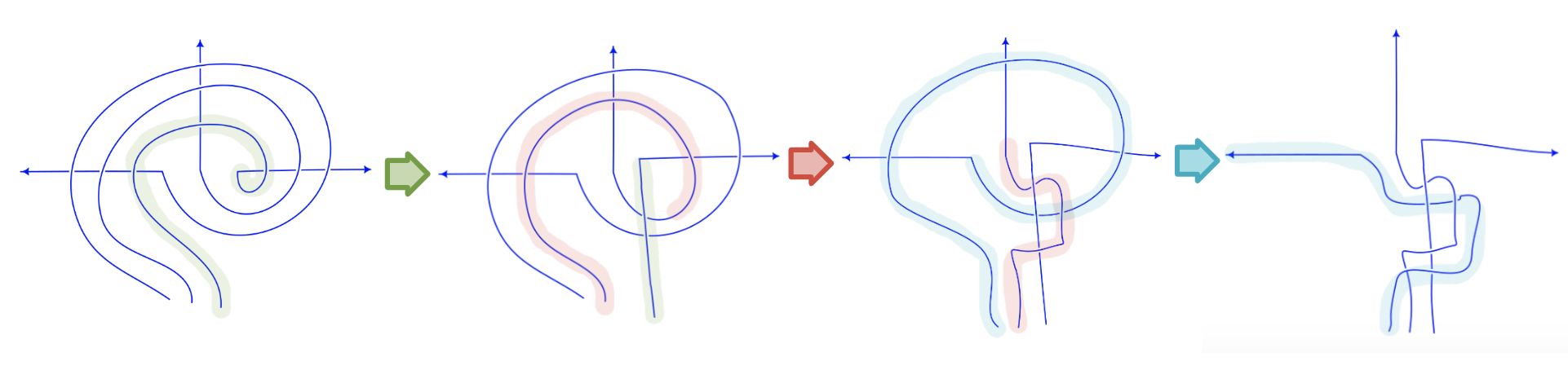}
\end{center}
\caption{From bridge diagram to full twist on three strands}\label{simplifytwisted}
\end{figure}
\begin{figure}[h]
\centering
\begin{center}
\includegraphics[scale=0.2]{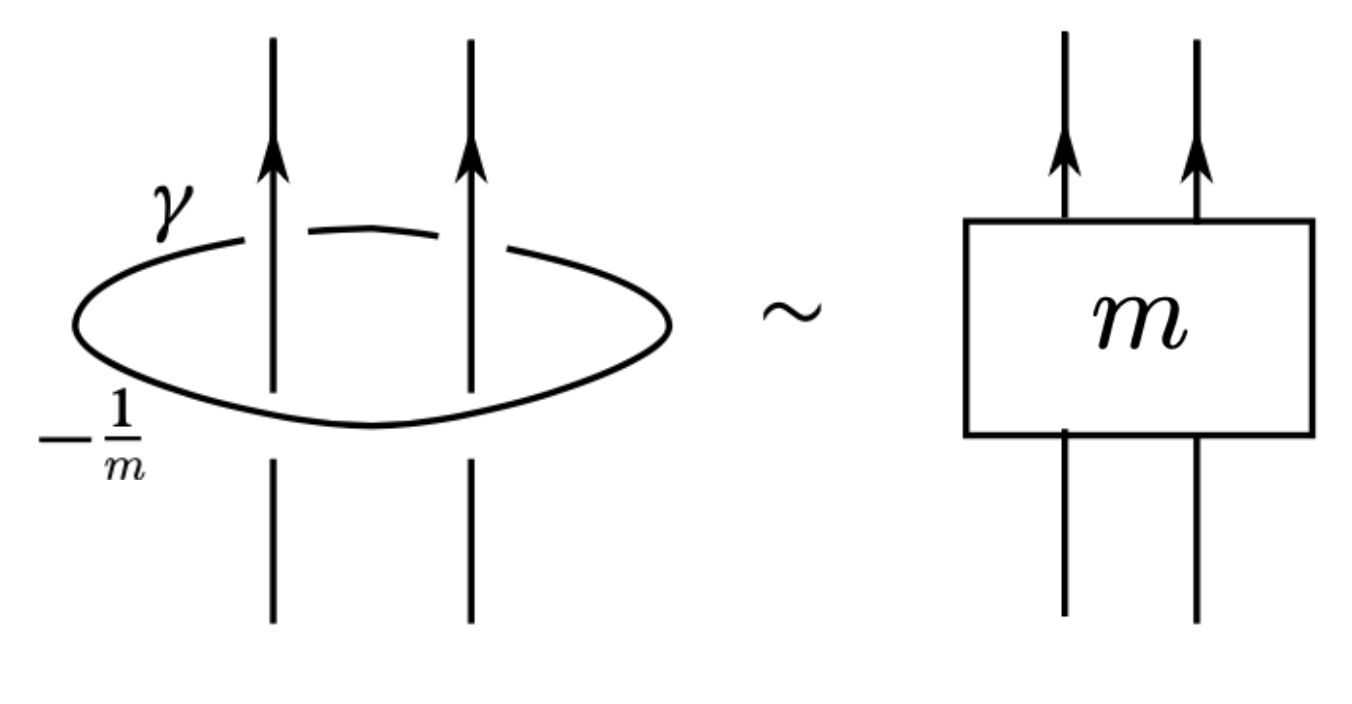}
\end{center}
\caption{Surgery description of twisting operation}\label{surgerytwist}
\end{figure}
\begin{figure}[h]
\centering
\begin{center}
\includegraphics[scale=0.4]{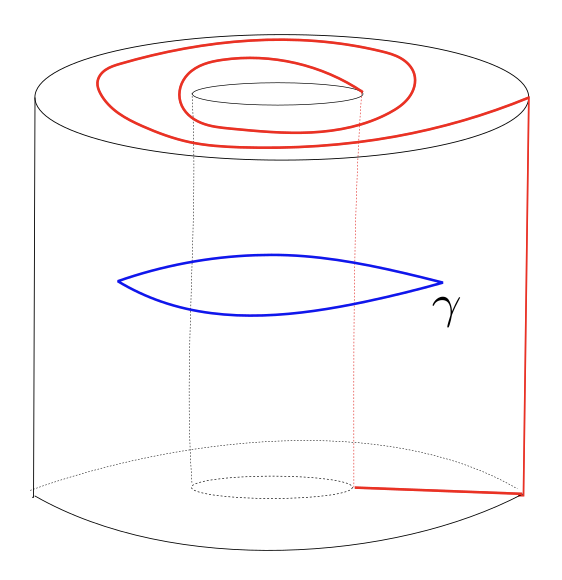}
\end{center}
\caption{$m-$th power of Dehn twist on the boundary and meridian of $-\frac{1}{m}$ surgery}\label{Dehntwistsurgery}
\end{figure}

There is a standard process to turn a bridge diagram of a knot into a multipointed Heegaard diagram. For $i=1,\dots,T$, let $\alpha_i$ be the boundary of a tubular neighborhood of the arc $a_i$ and $\beta_i$ be the boundary of a tubular neighborhood of $b_i$. Set $\bz=(z_0,z_1,\dots,z_T)$; $\bw=(w_0,w_1,\dots,w_T)$; $\balpha=(\alpha_1,\dots,\alpha_N)$; and $\bbeta=(\beta_1,\dots,\beta_N)$. Let $\mathcal{H}_m:=(S^2,\balpha,\bbeta,\bz,\bw)$ denote the multipointed Heegaard diagram constructed from the aforementioned bridge diagram of $K_m$.\\

Let $T_{\gamma} :S^2 \rightarrow S^2$ denote the positive Dehn twist along $\gamma$ and let $T_{\gamma}^{*} : Sym^{N}(S^2) \rightarrow Sym^{N}(S^2)$ be the induced map. Based on the construction of the bridge diagrams, the following relation holds between $\mathcal{H}_m$ and $\mathcal{H}_{m+1}$. If $\mathcal{H}_m=(S^2,\balpha,\bbeta,\bz,\bw)$ then $$\mathcal{H}_{m+1}=(S^2,\balpha, (T_{\gamma})^{-1} \bbeta,\bz,\bw)$$
This can be seen in Figures \ref{Heegaard3strandtwist} and \ref{Heegaard3strandtwist+}. The Figures depict the examples for the case of 3 strand twist i.e. $n=3$.\\

The negative Dehn twist introduces $(n-1)^2$ new 4-tuples to the intersection points between $\alpha_1,\dots,\alpha_{n-1}$ and $\beta_1,\dots,\beta_{n-1}$, one 4-tuple for each pair $\alpha_i$ and $\beta_j$. It doesn't change any of the other intersections. For example in  Figures \ref{Heegaard3strandtwist} and \ref{Heegaard3strandtwist+}, the intersection points of $\balpha$ and $\bbeta$ only change by addition of the four 4-tuples with subscript 2. Thus, there is a natural set injection  $$\mathbb{T}_{\alpha} \cap \mathbb{T}_{\beta}=\mathfrak{G}(\mathcal{H}_m)  \xrightarrow[]{t_m} \mathfrak{G}(\mathcal{H}_{m+1})=\mathbb{T}_{\alpha} \cap (T^{*}_{\gamma})^{-1}\mathbb{T}_{\beta}$$
If $\bx \in \mathfrak{G}(\mathcal{H}_m)$, let $\tilde{\bx} := t_m(\bx) \in \mathfrak{G}(\mathcal{H}_{m+1})$. The map $(T^{*}_{\gamma})^{-1}$ also induces a bijective map on Whitney disks from $\pi_2(\bx,\by)$ to $\pi_2(\tilde{\bx},\tilde{\by})$. If $\phi \in \pi_2(\bx,\by)$ is a Whitney disk, let $\tilde{\phi}$ denote the corresponding Whitney disk.\\

We partition the generators $\mathfrak{G}(\mathcal{H}_m)$ into $n$ families according to their vertices along $\alpha_1,\dots,\alpha_n$. The study of the effect of the map $t_m$ on these families and their grading is the main idea behind the proof.\\

In the rest, we first give a detailed proof of the case $n=3$. We will state the main steps of the proof for general $n \geq 3$ at the end of Section \ref{ProofSection}. Most of the arguments in the case of three strands directly generalise to $n \geq 3$. We use a simplified notation in the proof of $n=3$ case and in the figures. The general labelling is defined in Equations \ref{genlabelgood} and \ref{genlabelbad}.\\

\begin{figure}[h]
\centering
\begin{center}
\includegraphics[scale=0.5]{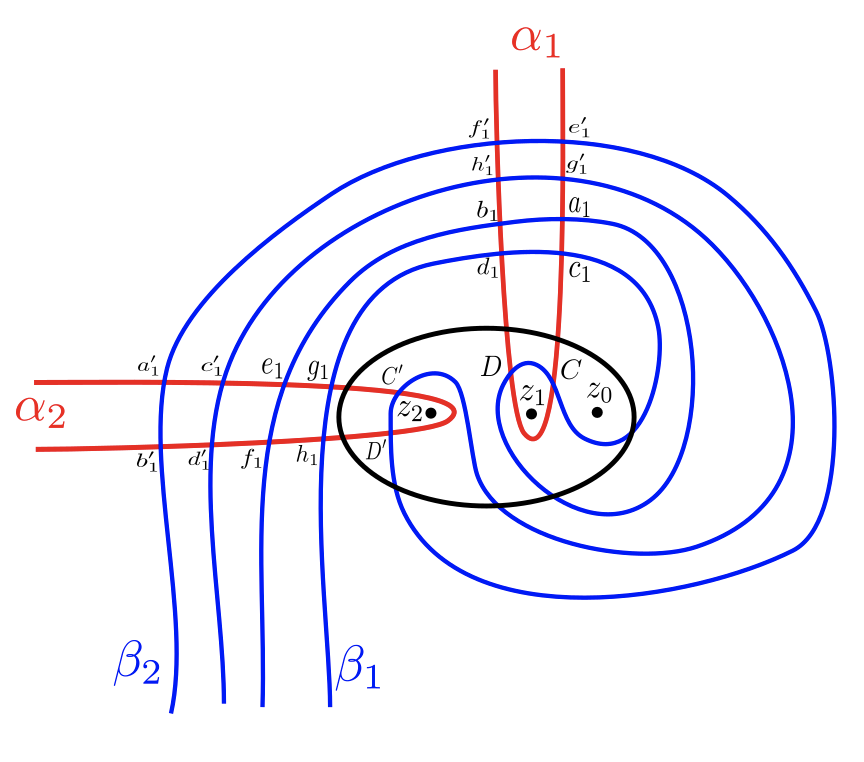}
\end{center}
\caption{Heegaard diagram for one full twist on 3 strands}\label{Heegaard3strandtwist}
\end{figure}
\begin{figure}[h]
\centering
\begin{center}
\includegraphics[scale=0.7]{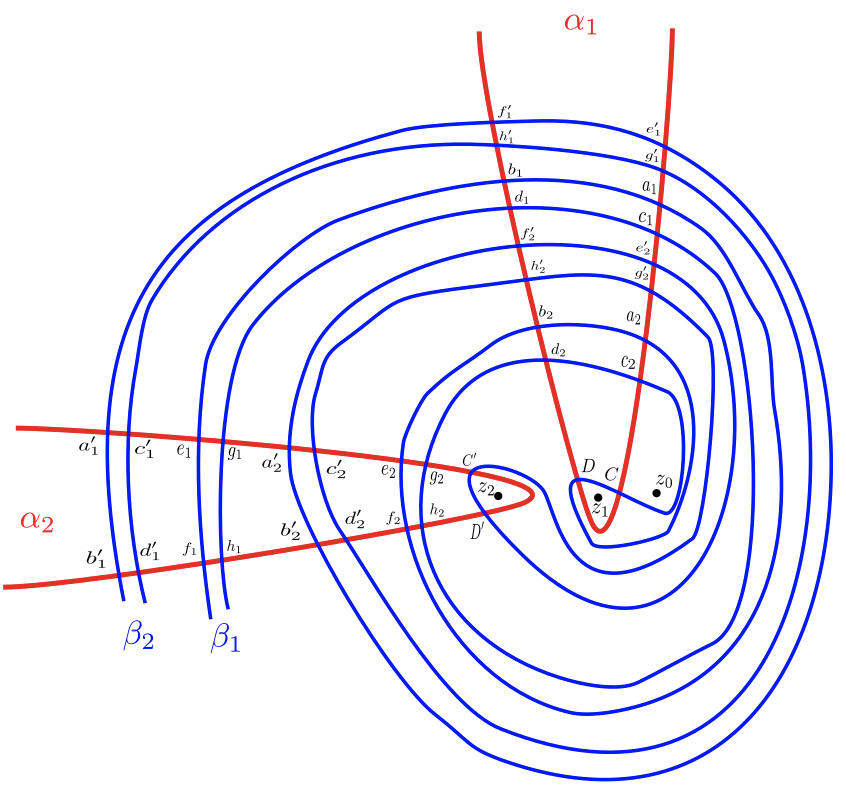}
\end{center}
\caption{Heegaard diagram for two full twist on 3 strands}\label{Heegaard3strandtwist+}
\end{figure}

The diagrams $\mathcal{H}_1$ and $\mathcal{H}_2$ of the three strand case can be seen in Figures \ref{Heegaard3strandtwist} and \ref{Heegaard3strandtwist+}. The intersection points between $\alpha_1,\alpha_2$ and $\beta_1,\beta_2$ in $\mathcal{H}_m$ are labeled as follows:  

$$\alpha_1 \cap \beta_1 \ = \ \bigcup_{i=1,\dots,m} \{a_i,b_i,c_i,d_i\} \cup \{C,D\} \ \text{ordered from top to bottom,}$$
$$\alpha_2 \cap \beta_2 \ = \ \bigcup_{i=1,\dots,m}\{a'_i,b'_i,c'_i,d'_i\} \cup \{C',D'\}\ \text{ordered from left to right,}$$ 
$$\alpha_2 \cap \beta_1 \ = \ \bigcup_{i=1,\dots,m}\{e_i,f_i,g_i,h_i\} \ \text{ordered from left to right,}$$ 
$$\alpha_1 \cap \beta_2 \ = \ \bigcup_{i=1,\dots,m}\{e'_i,f'_i,g'_i,h'_i\} \ \text{ordered from top to bottom.}$$

Let $\bx=(v_1,v_2,\dots,v_{N})$ be a generator where $v_i \in \alpha_i$. Define the following three families.\\

\noindent
In the first family, the vertices along $\alpha_1, \alpha_2$ comes from $\{C,D,C',D'\}$ i.e.
  $$\mathfrak{G}^{0}(\mathcal{H}_m):=\{\bx=(v_1,v_2,\dots,v_{N}) \ | \ v_1\in\{C,D\}, \ v_2\in\{C',D'\}\}.$$
  
\noindent 
In the second family, exactly one of the vertices along $\alpha_1$ and $\alpha_2$ comes from $\{C,D,C',D'\}$. It has the following two subfamilies, which are defined as follows:  
    $$\mathfrak{G}^{1}(\mathcal{H}_m)_{\{1\}}:=\{\bx=(v_1,v_2,\dots,v_{N}) \ | \ v_1\in\{C,D\}, \ v_2 \not\in\ \{C',D'\} \},$$
    $$\mathfrak{G}^{1}(\mathcal{H}_m)_{\{2\}}:=\{\bx=(v_1,v_2,\dots,v_{N}) \ | \ v_1\not\in\{C,D\}, \ v_2 \in\ \{C',D'\} \},$$
    $$\mathfrak{G}^{1}(\mathcal{H}_m)=\mathfrak{G}^{1}(\mathcal{H}_m)_{\{1\}} \  \cup\  \mathfrak{G}^{1}(\mathcal{H}_m)_{\{2\}}.$$
\noindent
In the third family, both vertices along $\alpha_i$ don't come from $\{C,D,C',D'\}$,
$$\mathfrak{G}^{2}(\mathcal{H}_m):=\{\bx=(v_1,v_2,\dots,v_{N}) \ | \ v_1\not\in\{C,D\}, \ v_2\not\in\{C',D'\}\}.$$
 Clearly, this is a decomposition i.e.
 $$\mathfrak{G}(\mathcal{H}_m)=\mathfrak{G}^{0}(\mathcal{H}_m) \bigsqcup \mathfrak{G}^{1}(\mathcal{H}_m) \bigsqcup \mathfrak{G}^{2}(\mathcal{H}_m).$$
 Note that the three families are similar to the families defined in Lemma 2.5. of Lambert-Cole\cite{LambertCole2016TwistingMA}. The only difference is that we suppress the distinction between \textit{twist} and \textit{negative} generators in Lambert-Cole's notation as it is not necessary for our purpose. \\
 
Now we prove Lemma \ref{Gradingchange} which is the analogue of a lemma of Lambert-Cole\cite{LambertCole2016TwistingMA}.
\begin{lemm}\label{Gradingchange}
    Choose $\bx,\by \in \mathfrak{G}(\mathcal{H}_m)$ and $\phi \in \pi_2(\bx,\by)$. Let $\tilde{\bx},\tilde{\by}$ be the corresponding generators in $\mathfrak{G}(\mathcal{H}_{m+1})$ and $\tilde{\phi}$ the corresponding Whitney disk in $\pi_2(\tilde{\bx},\tilde{\by})$.\\

\noindent (0) If $\bx \in \mathfrak{G}^{i}(\mathcal{H}_m)$ for $i=0,1,2$ then $\tilde{\bx} \in \mathfrak{G}^{i}(\mathcal{H}_{m+1})$ i.e. $t_m : x \rightarrow \tilde{x}$ respects the decomposition of $\mathfrak{G}(\mathcal{H}_m)$. Furthermore, $t_m$ respects the decomposition of $\mathfrak{G}^{1}(\mathcal{H}_{m})$ as well.\\
    
     \noindent (1) If $\bx,\by \in \mathfrak{G}^{i}(\mathcal{H}_m)$ for $i=0,2$ then 
    $$n_{\bz}(\tilde{\phi})=n_{\bz}(\phi), \ n_{\bw}(\tilde{\phi})=n_{\bw}(\phi), \ \mu(\tilde{\phi})= \mu(\phi).$$
    $$\text{Thus} \ A(\tilde{\bx})-A(\tilde{\by})=A(\bx)-A(\by), \ M(\tilde{\bx})-M(\tilde{\by})=M(\bx)-M(\by).$$
    If $i=1$ then the same holds if $\bx,\by \in \mathfrak{G}^{1}(\mathcal{H}_m)_{\{1\}}$ or $\bx,\by \in \mathfrak{G}^{1}(\mathcal{H}_m)_{\{2\}}$.\\

    \noindent(2) If $\bx \in \mathfrak{G}^{1}(\mathcal{H}_m), \ \by \in \mathfrak{G}^{0}(\mathcal{H}_m)$ then 
    $$n_{\bz}(\tilde{\phi})=n_{\bz}(\phi)+3, \ n_{\bw}(\tilde{\phi})=n_{\bw}(\phi), \ \mu(\tilde{\phi})= \mu(\phi)+2.$$
    $$\text{Thus} \ A(\tilde{\bx})-A(\tilde{\by})=A(\bx)-A(\by)+3, \ M(\tilde{\bx})-M(\tilde{\by})=M(\bx)-M(\by)+2.$$
    (3) If $\bx \in \mathfrak{G}^{2}(\mathcal{H}_m), \ \by \in \mathfrak{G}^{1}(\mathcal{H}_m)$ then
    $$n_{\bz}(\tilde{\phi})=n_{\bz}(\phi)+3, \ n_{\bw}(\tilde{\phi})=n_{\bw}(\phi), \ \mu(\tilde{\phi})= \mu(\phi)+2.$$
    $$\text{Thus} \ A(\tilde{\bx})-A(\tilde{\by})=A(\bx)-A(\by)+3, \ M(\tilde{\bx})-M(\tilde{\by})=M(\bx)-M(\by)+2.$$
    (4) If $\bx \in \mathfrak{G}^{2}(\mathcal{H}_m), \ \by \in \mathfrak{G}^{0}(\mathcal{H}_m)$ then 
    $$A(\tilde{\bx})-A(\tilde{\by})=A(\bx)-A(\by)+6, \ M(\tilde{\bx})-M(\tilde{\by})=M(\bx)-M(\by)+4.$$
\end{lemm}
\begin{proof}
The transition from $\mathcal{H}_m$ to $\mathcal{H}_{m+1}$ only adds four new 4-tuples of intersection points, one to each $\alpha_i \cap \beta_j \ \ i,j \in \{1,2\}$. It doesn't change any of the other intersection points in $\mathcal{H}_m$, and as a result, transformation $t_m: x \rightarrow \tilde{x}$ respects the decomposition of $\mathfrak{G}(\mathcal{H}_m)$ and $\mathfrak{G}(\mathcal{H}_{m+1})$.\\

Let $D(\phi)$ be the domain in $\mathcal{H}_m$ corresponding to $\phi$. Let the Heeggaard state $\bx=(v_1,\dots,v_N)$ and $\by=(w_1,\dots,w_N)$ where $v_i,w_i \in \alpha_i$. Orient $\gamma$ counterclockwise, which means that the disk containing $z_0,z_1\text{ and }z_2$ is inside of $\gamma$. Note that $\partial D(\phi) \cap \bbeta = \bx-\by$, and $v_i$ or $w_i$ can be inside $\gamma$  if and only if $i=1,2$. In case (1), $v_i,w_i$ for $i=1,2$ are either both inside $\gamma$ or both outside $\gamma$ i.e. $\gamma$ doesn't separate $v_i,w_i$ for $i=1,2$. As a result, the algebraic intersection of the $\beta-$components of $\partial D(\phi)$ with $\gamma$ is $0$. Thus, the intersection numbers $n_z$ and $n_w$ and the Maslov index $\mu$ are unchanged by the Dehn twist. The conclusion about the Alexander and Maslov gradings follows from their definitions.\\

We can use the statement of the case (1) to reduce the proof of other cases to special pairs of Heegaard states and special Whitney disks.  This follows from the additivity of $n_z$, $n_w$ and $\mu$ under the composition of Whitney disks 
$$\pi_2(\bx,\bxs) \times \pi_2(\bxs,\bys) \times \pi_2(\bys,\by) \rightarrow \pi_2(\bx,\by).$$
In case (2), we can let $\bys=(C,C',v_{s_3},\dots,v_{s_N})$ be any Heegaard state in $\mathfrak{G}^{0}(\mathcal{H}_m)$ containing intersection points $C,C'$. Consider the Heegaard state $\bxs=(c_n,C',v_{s_3},\dots,v_{s_N})) \in \mathfrak{G}(\mathcal{H}_m)_{\{2\}} \subset  \mathfrak{G}^{1}(\mathcal{H}_m)$. Let $D(\phi_s)$ be the bigon connecting $c_n$ and $C$. Due to the definition of $\bxs$ and $\bys$, the bigon $D(\phi_s)$ is the domain of a Whitney disk  $\phi_s \in \pi_2(x_s,y_s)$. Now we need to look at the domain $D(\tilde{\phi_s})$ in $\mathcal{H}_{m+1}$. Since the Dehn twist is a local operation we only need to analyze the case $m=1$ and the result will hold in the general case. This is done in Figure \ref{domaintilde}. We can see that $D(\tilde{\phi_s})$ will be the sum of the bigon connection $c_{m+1}$ and $C$ and the bigon connecting $c_m$ and $c_{m+1}$. The second bigon contains $z_0,z_1,z_2$ and none of the basepoints in $\bw$. Furthermore, the second bigon has one acute right-angled corner and one obtuse right-angled corner (similar to Figure \ref{bigon}). Lipshitz's combinatorial formula for the Maslov index will give us that $\mu(\tilde{\phi_s})=3=\mu(\phi_s)+2$. This leads to the desired conclusion.\\

We can go through the same process for state $\bxs=(C,c'_n,v_{s_3},\dots,v_{s_N}) \in \mathfrak{G}(\mathcal{H}_m)_{\{1\}} \subset  \mathfrak{G}^{1}(\mathcal{H}_m)$ to complete the proof of case (2). The proof of case (3) is similar and case (4) follows from case (2) and (3) using the composition of Whitney disks.\\
\begin{figure}[h]
\centering
\begin{center}
\includegraphics[scale=0.5]{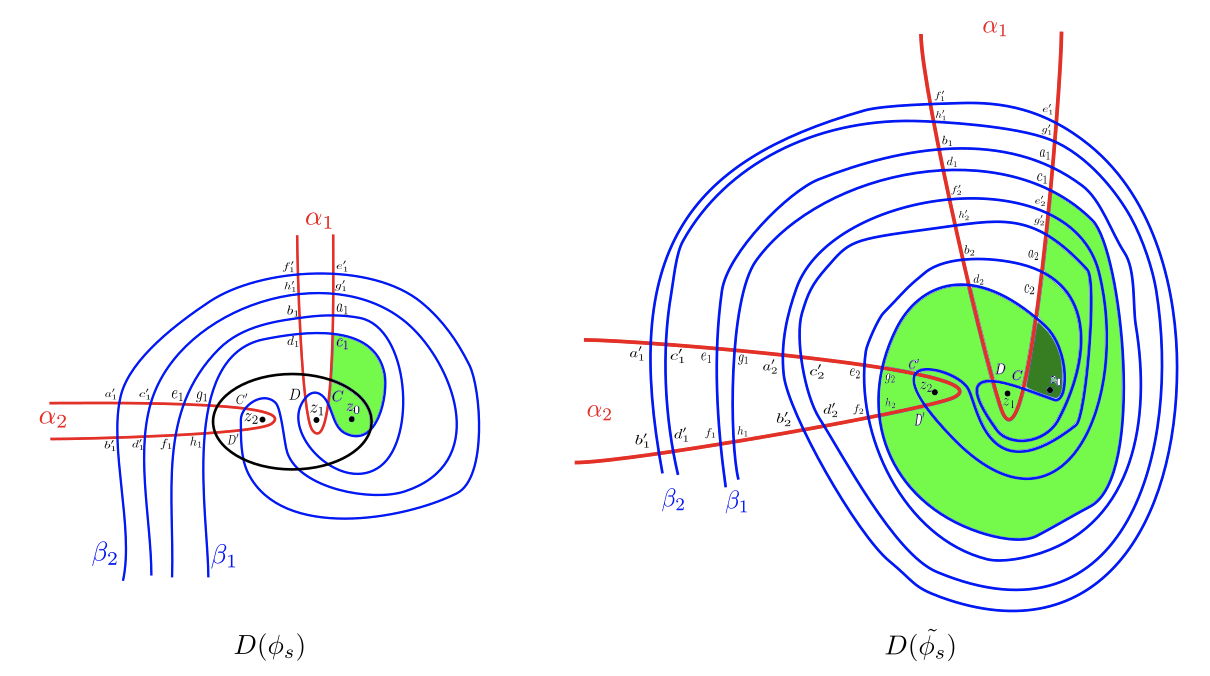}
\end{center}
\caption{the domains $D(\phi_s)$ and $D(\tilde{\phi_s})$ (The light green regions have multiplicity $+1$ while the dark green region has multiplicity $+2$)}\label{domaintilde}
\end{figure}
\end{proof}
Let us set a notation before starting the proof of Theorem \ref{MainTheorem}. Let $t^{m+1}_{m} := t_m : \mathfrak{G}(\mathcal{H}_m) \rightarrow \mathfrak{G}(\mathcal{H}_{m+1})$. We need to define a notation for the compositions of these maps as follows 
$$\forall \ m_1 \leq m_2 \ \ t^{m_2}_{m_1}:= t_{m_2-1} \circ \dots \circ t_{m_1+1} \circ t_{m_1} : \mathfrak{G}(\mathcal{H}_{m_1}) \rightarrow \mathfrak{G}(\mathcal{H}_{m_2}).$$
We are now ready to prove Theorem \ref{MainTheorem}. 
\begin{proof}[Proof of Theorem \ref{MainTheorem}]
We want to use map $t_m$ to construct the stabilization isomorphisms. The main issue is that $t_m$ isn't surjective on $\mathfrak{G}^{0}(\mathcal{H}_m)$ and $\mathfrak{G}^{1}(\mathcal{H}_m)$. Any generator $\bx_{m+1} \in \mathfrak{G}(\mathcal{H}_{m+1}) \setminus t_m(\mathfrak{G}(\mathcal{H}_{m}))$ includes an intersection points $s_{m+1}$ for $s \in \{a,b,c,d,e,f,g,h\} \cup \{a',b',c',d',e',f',g',h'\}$. These generators can never take the maximum Alexander grading, because there is a bigon connection $s_{m}$ to $s_{m+1}$ with $n_z=3$ and $n_w=0$ (this bigon can be seen in Figure \ref{bigon}). This means that  Heegaard state $\bx_{m} = \bx_{m+1} \setminus s_{m+1} \cup s_{m}$ satisfies the equation 
$$A(\bx_{m}) - A(\bx_{m+1}) = 3.$$
This means that elements of $\mathfrak{G}(\mathcal{H}_{m+1}) \setminus t_m(\mathfrak{G}(\mathcal{H}_{m}))$ are not in the top three Alexander gradings of $\mathfrak{G}(\mathcal{H}_{m+1})$. In fact, one can define ${\bx_{1} = \bx_{m+1} \setminus s_{m+1} \cup s_{1}}$ which satisfies 
$$A(\bx_{1}) - A(\bx_{m+1}) = 3m.$$
which means elements of $\mathfrak{G}(\mathcal{H}_{m+1}) \setminus t_m(\mathfrak{G}(\mathcal{H}_{m}))$ are not in the top $3m$ Alexander gradings of $\mathfrak{G}(\mathcal{H}_{m+1})$.\\

Let $L_m,R_m$ denote the lowest and highest Alexander degree of a generator in $\widehat{CFK}(\mathcal{H}_m)$. We claim that for any $k \in \mathbb{N}$ if $m$ is high enough then 
the top $k$ groups in the chain complex (with respect to Alexander grading) are generated by elements in $\mathfrak{G}^{2}(\mathcal{H}_m)$ i.e. 
\begin{equation}\label{saturation}
    \forall \  0 \leq j \leq k-1 \ \ \widehat{CFK}(\mathcal{H}_m,R_m-j) \subset \langle  \mathfrak{G}^{2}(\mathcal{H}_m)\rangle.
\end{equation}
First of all, if $3m \geq k$, due to the argument of the last paragraph, the top $k$ Alexander gradings of $\mathfrak{G}(\mathcal{H}_{m+1})$ will be generated by elements in the image of $t_m$. Due to Lemma \ref{Gradingchange}, the (signed) distance of the Alexander gradings of generators inside $\mathfrak{G}^{2}(\mathcal{H}_m)$ and other generators grows linearly under the map $t_m$. Furthermore $t_m$ respects the decomposition. As a result, there exist $m_0$ such that at least one of the generators of $\underset{j=0,\dots,k-1}\bigcup \widehat{CFK}(\mathcal{H}_{m_0},R_{m_0}-j)$ is inside $t_{m_0-1}(\mathfrak{G}^{2}(\mathcal{H}_{m_0-1})) \subset \mathfrak{G}^{2}(\mathcal{H}_{m_0})$ i.e. 
$$\exists \  \bx_0 \in \mathfrak{G}^{2}(\mathcal{H}_{m_0-1}) \ \text{s.t.} \ \widetilde{\bx_0} = t_{m_0-1}(\bx_0) \in \underset{j=0,\dots,k-1}\bigcup \widehat{CFK}(\mathcal{H}_{m_0},R_{m_0}-j)$$

We claim that Equation \ref{saturation} holds for $m\geq m_0+\lceil\frac{2k}{3}\rceil$. This again comes from Lemma \ref{Gradingchange} and the linear growth of the relative Alexander grading with slope at least $3$ under the map $t_m$. To be precise, for any $\by_0 \in \mathfrak{G}(\mathcal{H}_{m_0}) \setminus \mathfrak{G}^{2}(\mathcal{H}_{m_0})$, we have the following inequality
$$A(t^{m}_{m_0}(\widetilde{\bx_0})) - A(t^{m}_{m_0}(\by_0)) \geq 3(m-m_0) + A(\widetilde{\bx_0}) - A(\by_0) \geq 2k + A(\widetilde{\bx_0}) - A(\by_0),$$
which comes from repeated use of Lemma \ref{Gradingchange}. If ${\by_0 \in \underset{j=0,\dots,k-1}\bigcup \widehat{CFK}(\mathcal{H}_{m_0},R_{m_0}-j)}$ then $A(t^{m}_{m_0}(\widetilde{\bx_0})) - A(t^{m}_{m_0}(\by_0)) \geq k$ which means $t^{m}_{m_0}(\by_0)$ is not in the top $k$ Alexander gradings of $\widehat{CFK}(\mathcal{H}_{m})$. Since this holds for any generator $\by_0 \in \mathfrak{G}(\mathcal{H}_{m_0}) \setminus \mathfrak{G}^{2}(\mathcal{H}_{m_0})$, this is equivalent to Equation \ref{saturation}.\\

Now let $m' \geq m$. Due to the argument in the first paragraph, ${t^{m'}_{m}:\mathfrak{G}(\mathcal{H}_{m}) \rightarrow \mathfrak{G}(\mathcal{H}_{m'})}$ is a bijection in top $3m'$ Alexander gradings. As a result, the map $t^{m'}_{m}$ is also a bijection on top $k$ Alexander gradings. Furthermore, due to Part (1) of Lemma \ref{Gradingchange} and the argument in the previous paragraph, the map $t^{m'}_{m}$ also preserves relative Alexander and Maslov gradings. The induced map 
$$t^{m'}_{m} \ : \ \underset{j=0,\dots,k-1}\bigoplus \widehat{CFK}(\mathcal{H}_{m},R_{m}-j) \longrightarrow \underset{j=0,\dots,k-1}\bigoplus \widehat{CFK}(\mathcal{H}_{m'},R_{m'}-j)$$
is an isomporphism of Abelian groups and decomposes to $k$ components with respect to Alexander grading i.e. 
$$t^{m'}_{m}(j) \ : \ \widehat{CFK}(\mathcal{H}_{m},R_{m}-j) \longrightarrow \widehat{CFK}(\mathcal{H}_{m'},R_{m'}-j).$$
Now, we prove that these maps are also chain maps. Fix $\bx,\by \in \mathfrak{G}^2_{m}$ and some $\phi \in \pi_2(\bx,\by)$ with $n_z(\phi)=0$ and $\mu(\phi)=1$. We can choose an open neighborhood $W$ of $\overline{D(\phi)}$ to be disjoint from the curve $\gamma$. This is due to the fact that any domain intersecting $\gamma$ but not $\bz$ must have a corner inside $\gamma$ i.e. in one of the points $C,D,C',D'$ (See Figure \ref{corners}). Thus we can assume that the support of $T_{\gamma}$ is disjoint from $W$ and that the support of $T_{\gamma}^{*}$ is disjoint from $Sym^N(W)$. Let $u \in \mathcal{M}_{J_s}(\phi)$ be a holomorphic representative. Using the Localization Principle of Rasmussen (Lemma 9.9, \cite{Rasmussenthesis}), we can deduce that the image of $u$ is contained in ${Sym^N(W) \subset Sym^N(S^2)}$. Since $T_{\gamma}^*$ is the identity on $Sym^N(W)$, there will be a one-to-one bijection between $\mathcal{M}_{J_s}(\phi)$ and $\mathcal{M}_{J_s}(\tilde{\phi})$. After quotienting the action of $\mathbb{R}$, we will have $$t_m(\widehat{\partial}\bx)=\widehat{\partial}\tilde{\bx}=\widehat{\partial}t_m(\bx).$$ One can easily deduce that $t^{m'}_{m}$ restricted to top $k$ Alexander gradings is also a chain map.\\

The induced maps on the homology are also isomorphisms i.e. 
$$(t^{m'}_{m}(j))_{*} \ : \ \widehat{HFK}(K_m,R_{m}-j) \xrightarrow{\qquad \cong \qquad}\widehat{HFK}(K_{m'},R_{m'}-j).$$
These isomorphisms are almost the desired statement of Theorem \ref{MainTheorem}. The only remaining issue is the comparison of $R_m$ and $g(K_m)$. Without this part, all these isomorphisms might be between trivial groups. In fact, a direct comparison of $R_m$ and $g(K_m)$ seem not to be possible, as we are not aware of any bounds for $R_m$. However, it is possible to compare their growth.\\

The growth of genus is known by the work of Baker and Montegi (Theorem 2.1., \cite{Baker2017SeifertVS}). They prove that the genus of $K_m$ grows linearly with respect to $m$ with the slope $\frac{n(n-1)}{2}$ (in the general case when the twisting happens on $n$ parallel strands). This coincides with the maximum slope growth of relative Alexander grading (Lemma \ref{Gradingchange}). We continue our discussion on $n=3$, but this easily extends to the general case.\\

Notice that we can find a bound for the growth of the stabilization region. The main condition for $m$ in the above argument was Equation \ref{saturation} i.e. $${\underset{j=0,\dots,k-1}\bigoplus \widehat{CFK}(\mathcal{H}_{m},R_{m}-j)} \ \subset \ \langle \mathfrak{G}^{2}(\mathcal{H}_m) \rangle.$$ We argued that this condition would result in the stabilization of the top $k$ homology groups with respect to Alexander grading. Again due to the linear growth of the relative Alexander grading in Lemma \ref{Gradingchange} and our discussion about the Alexander grading of elements outside the image of $t_m$, we have 
\begin{equation*}
\begin{split}
    & \underset{j=0,\dots,k-1}\bigoplus \widehat{CFK}(\mathcal{H}_{m},R_{m}-j) \subset \langle \mathfrak{G}^{2}(\mathcal{H}_m) \rangle \\
    \Longrightarrow &\underset{j=0,\dots,k+2}\bigoplus \widehat{CFK}(\mathcal{H}_{m+1},R_{m+1}-j) \subset  \langle t_m(\mathfrak{G}^{2}(\mathcal{H}_{m}))\rangle \subset \langle \mathfrak{G}^{2}(\mathcal{H}_{m+1})\rangle.
\end{split}
\end{equation*}
This means that the stabilization region grows by (at least) a linear function with slope $3$.\\

Now let $m_1=m+d$ and assume that $R_{m_1}=R_{m}+3d+r_{m_1}$. We know that the top $k+3d$ homology groups are stabilised after applying $m_1$ twists i.e. all of the Alexander gradings higher than $R_{m}+r_{m_1}-k+1$ are stabilized.\\ 

We also know that $g(K_{m_1})=g(K_m)+3d$. Note that based on Lemma \ref{Gradingchange} we have 
$$(R_{m_1}-L_{m_1}) \leq (R_m-L_m) +6d.$$
As a result, we have 
\begin{equation*}\label{thepull}
    L_m -L_{m_1} \leq 3d-r_{m_1}.
\end{equation*}
On the other hand, we know that $\widehat{HFK}(K_{m_1},-g(K_{m_1}))$ is non-trivial. As a result, $L_{m_1} \leq -g(K_{m_1}) = -g(K_m) -3d$.  In combination with Equation \ref{thepull}, we can deduce
$$L_m+g(K_m) \leq -r_{m_1} \Longrightarrow r_{m_1} \leq -L_m -g(K_m).$$
This upper bound on $r_{m_1}$ is independent of $m_1$, as a result, we have an upper bound on $R_{m}+r_{m_1}-k+1$, which is independent of $m_1$. Now note that for large enough $m_1$, all of the knot Floer homology groups $$\widehat{HFK}(K_{m_1},j)\ \text{for} \  j \in [R_{m}-L_m -g(K_m)-k+1 , g(K_{m_1})]$$ are stabilised. This shows that the stabilisation phenomena reaches the non-trivial extremal knot Floer homologies which is the main statement of Theorem.  \ref{MainTheorem}.\\
\begin{minipage}{\linewidth}
\begin{center}
\includegraphics[scale=0.4]{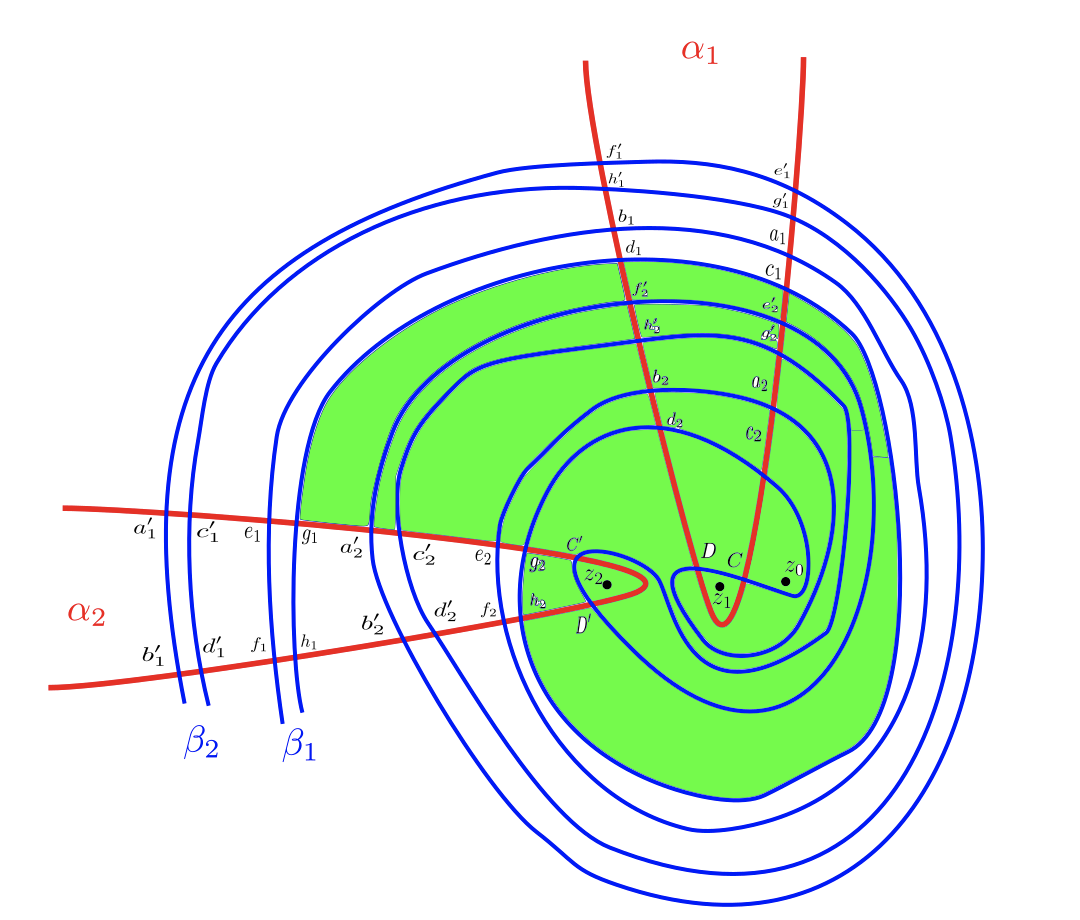}
\captionof{figure}{A bigon connecting $g_1$ to $g_2$}\label{bigon}
\end{center}
\end{minipage}
\begin{figure}[h]
\centering
\begin{center}
\includegraphics[scale=0.2]{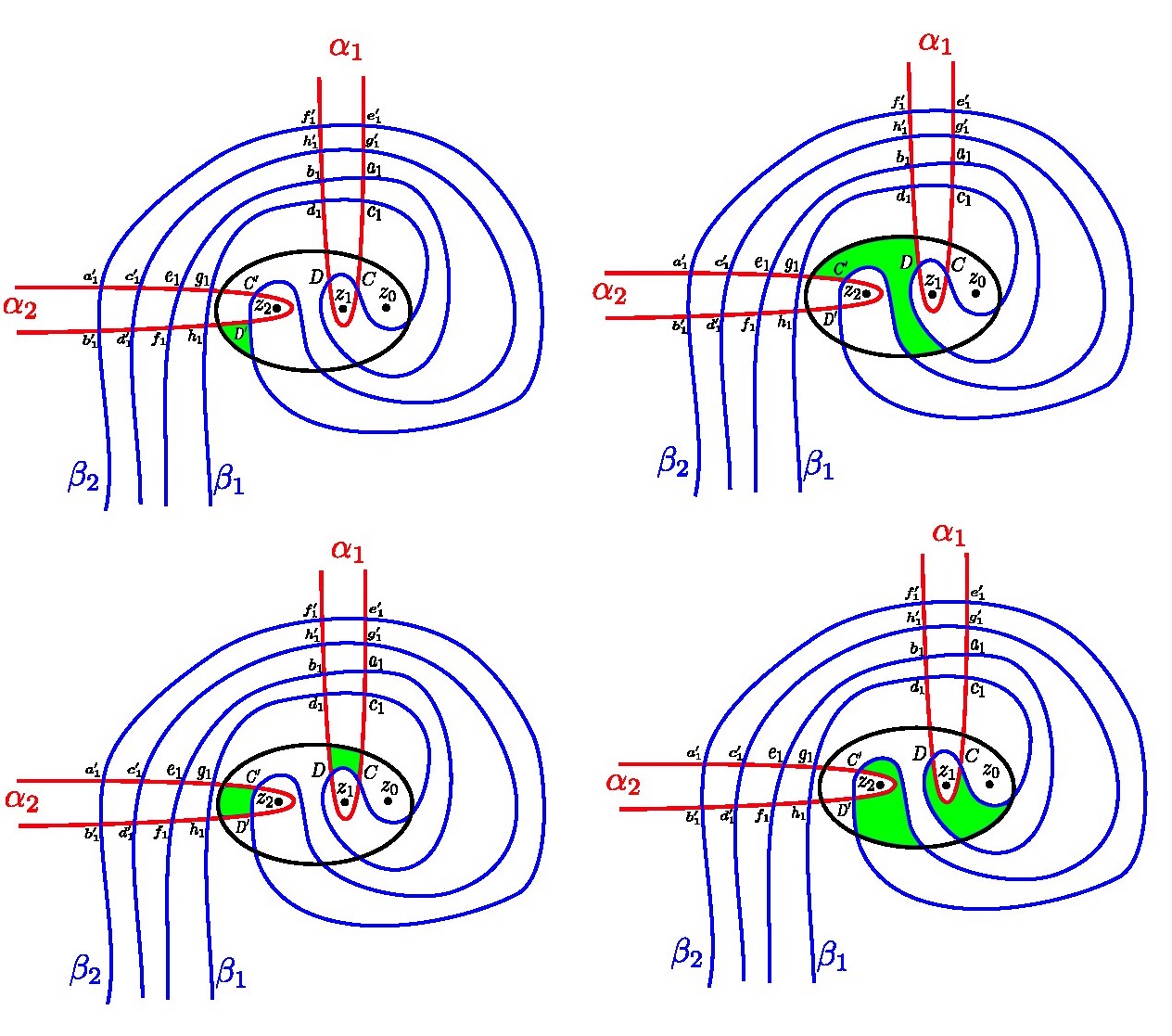}
\end{center}
\caption{Local picture of domains intersecting $\gamma$ but not~$\bz$}\label{corners}
\end{figure}

Now we state the main steps of the proof for the general case i.e. $n\geq 3$. First of all, note that we still can label the intersection points in $\alpha_i \cap \beta_j \subset \mathcal{H}_m$ as follows: 
\begin{equation}\label{genlabelgood}
        \forall \ 1\leq i\leq n-1  \ : \  \alpha_i \cap \beta_i \ = \ \bigcup_{k=1,\dots,m} \{a^{i}_k,b^{i}_k,c^{i}_k,d^{i}_k\} \cup \{C^{i},D^{i}\},
\end{equation}
\begin{equation}\label{genlabelbad}
       \forall \ 1\leq i<j \leq n-1  \ : \  \alpha_i \cap \beta_j \ = \ \bigcup_{k=1,\dots,m} \{e^{ij}_k,f^{ij}_k,g^{ij}_k,h^{ij}_k\},
\end{equation}
where $C^{i},D^{i}$ are vertices of the bigon formed between $\alpha_i$ and $\beta_i$ in the disk bound by $\gamma$ and contains the basepoint $z_i$.\\

Now we can proceed and decompose the elements of $\mathfrak{G}(\mathcal{H}_m)$ to $n$ families 
$$\mathfrak{G}(\mathcal{H}_m)=\bigsqcup_{0\leq i\leq n-1} \mathfrak{G}^{i}(\mathcal{H}_m),$$
where $\mathfrak{G}^{i}(\mathcal{H}_m)$ is the subset of Heegaard states which include $n-1-i$ intersection points from $\bigcup\limits_{1 \leq j \leq n-1} \{C^j ,D^j\}$. Note that this coincides with the decomposition constructed before Lemma \ref{Gradingchange}. One can further decompose each of these families as follows. Let $\bx=(v_1,\dots,v_N)$ be a Heegaard state where $v_i \in \alpha_i$, then
$$\mathfrak{G}^{i}(\mathcal{H}_m) =\bigsqcup\limits_{\substack{A \subset [n] \\ |A|=n-1-i}} \mathfrak{G}^{i}(\mathcal{H}_m)_{A} :=\bigsqcup\limits_{\substack{A \subset [n] \\ |A|=n-1-i}} \{ \bx \in \mathfrak{G}^{i}(\mathcal{H}_m) \ | \  \forall k \in A \ , \ v_k \in \{C^k,D^k\}\},$$
where $[n]:=\{0,1,\dots,n-1\}$.\\

The following lemma is the generalization of Lemma \ref{Gradingchange}.\\

\begin{lemm}\label{Gengradingchange}
    Choose $\bx,\by \in \mathfrak{G}^{i}(\mathcal{H}_m)$ and $\phi \in \pi_2(\bx,\by)$. Let $\tilde{\bx},\tilde{\by}$ be the corresponding generators in $\mathfrak{G}^{i}(\mathcal{H}_{m+1})$ and $\tilde{\phi}$ the corresponding Whitney disk in $\pi_2(\tilde{\bx},\tilde{\by})$ (the image of map $t_m$).\\

    \noindent (0) If $\bx \in \mathfrak{G}^{i}(\mathcal{H}_m)_{A}$, then $\tilde{\bx} \in \mathfrak{G}^{i}(\mathcal{H}_{m+1})_{A}$ i.e. $t_m : x \rightarrow \tilde{x}$ respects the decomposition of $\mathfrak{G}(\mathcal{H}_m)$.\\

    \noindent(1) If $\bx,\by \in \mathfrak{G}^{i}(\mathcal{H}_m)_{A}$, then
    $$n_{\bz}(\tilde{\phi})=n_{\bz}(\phi), \ n_{\bw}(\tilde{\phi})=n_{\bw}(\phi), \ \mu(\tilde{\phi})= \mu(\phi).$$
    $$\text{Thus} \ A(\tilde{\bx})-A(\tilde{\by})=A(\bx)-A(\by), \ M(\tilde{\bx})-M(\tilde{\by})=M(\bx)-M(\by).$$
\noindent(2) If $\bx \in \mathfrak{G}^{i+1}(\mathcal{H}_m)_{A'}$ and $\by \in \mathfrak{G}^{i}(\mathcal{H}_m)_{A}$ such that $A' \subset A$  and ${|A|=|A'|+1}$, then 
 $$n_{\bz}(\tilde{\phi})=n_{\bz}(\phi)+n, \ n_{\bw}(\tilde{\phi})=n_{\bw}(\phi), \ \mu(\tilde{\phi})= \mu(\phi)+2.$$
    $$\text{Thus} \ A(\tilde{\bx})-A(\tilde{\by})=A(\bx)-A(\by)+n, \ M(\tilde{\bx})-M(\tilde{\by})=M(\bx)-M(\by)+2.$$
\noindent(3) If $\bx \in \mathfrak{G}^{n-1}(\mathcal{H}_m)$ and $\by \in \mathfrak{G}^{0}(\mathcal{H}_m)$, then $$A(\tilde{\bx})-A(\tilde{\by})=A(\bx)-A(\by)+n(n-1), \ M(\tilde{\bx})-M(\tilde{\by})=M(\bx)-M(\by)+2(n-1).$$
\end{lemm}
The proof is similar to the proof of Lemma \ref{Gradingchange}. Case (0) just follows from the fact that Dehn twists along the curve $\gamma$ don't change the existing intersection points. Case (1) follows from the fact that the algebraic intersection of the $\beta-$ components of $\partial D(\phi)$ with $\gamma$ is $0$. Thus, the intersection numbers $n_{\bz}$ and $n_{\bw}$ and the Maslov index $\mu$ are unchanged by the Dehn twist. Case~(1) and the additivity of $n_{\bz}$, $n_{\bw}$ and $\mu$ under the composition of Whitney disks reduce Case (2) to special pairs of Heegaard states and special Whitney disks. Assume $A \setminus A' = \{h\}$. Similar to the Lemma \ref{Gradingchange}, this special pair $\bxs=(v_{s_1},\dots,v_{s_N})$ and $\bys=(w_{s_1},\dots,w_{s_N})$ are constructed by setting 
\begin{equation*}
    v_{s_h} = c^{h}_{m} \ , \ w_{s_h} = C^{h} \ , \text{and} \ 
    v_{s_j} =  w_{s_j}  \ \forall \ j \neq h
\end{equation*}
and taking $D(\phi_s)$ to be the bigon connecting them. Similar to Figure \ref{domaintilde}, this bigon only contains basepoint $z_h$ with multiplicity $1$. The domain $D(\tilde{\phi_s})$ in $\mathcal{H}_{m+1}$ will be the sum of the bigon connection $c^{h}_{m+1}$ and $C^{h}$ and the bigon connecting $c^{h}_m$ and $c^{h}_{m+1}$. The second bigon contains $z_0,z_1,\cdots,z_{n-1}$ and none of the basepoints in $\bw$. Furthermore, the second bigon has one acute right-angled corner and one obtuse right-angled corner (similar to Figure \ref{bigon}). The computations will follow similarly to the case of $n=3$. \\

Now the next step of the proof is to analyze generators ${\bx_{m+1} \in \mathfrak{G}(\mathcal{H}_{m+1}) \setminus t_m(\mathfrak{G}(\mathcal{H}_{m}))}$. Any such generator includes an intersection point in 
$$\bigcup_{1 \leq i \leq n-1}\{a^{i}_{m+1},b^{i}_{m+1},c^{i}_{m+1},d^{i}_{m+1}\} \ \ \bigsqcup \bigcup_{\ 1\leq i<j \leq n-1 }\{e^{ij}_{m+1},f^{ij}_{m+1},g^{ij}_{m+1},h^{ij}_{m+1}\}$$
Similar to Figure \ref{bigon}, there will be a bigon connecting $s_{m}$ to $s_{m+1}$ for all $s \in \{a^{i},b^{i},c^{i},d^{i},e^{ij},f^{ij},g^{ij},h^{ij}\}$. More generally bigons connecting $s_i$ to $s_{i+1}$ for $1 \leq i \leq m$, each containing basepoints $z_0, \dots, z_{n-1}$. This means that elements of $\mathfrak{G}(\mathcal{H}_{m+1}) \setminus t_m(\mathfrak{G}(\mathcal{H}_{m}))$ are not in the top $nm$ Alexander gradings of $\mathfrak{G}(\mathcal{H}_{m+1})$. This means that for $m' \geq m \geq \frac{k}{n} $, the map ${t^{m'}_{m}: \mathfrak{G}(\mathcal{H}_{m}) \rightarrow  \mathfrak{G}(\mathcal{H}_{m'})}$ is a bijection on top $k$ Alexander gradings.\\

The third step of the proof is to show that for high enough $m$, the top $k$ groups in the chain complex (with respect to Alexander grading) are generated by elements in $\mathfrak{G}^{n-1}(\mathcal{H}_m)$ i.e.
\begin{equation}\label{TopAlexandersaturatioon}
\forall \ 0 \leq j \leq k-1 \ \  \widehat{CFK}(\mathcal{H}_m,R_m-j) \subset \langle \mathfrak{G}^{n-1}(\mathcal{H}_m)\rangle
\end{equation}
The same argument works for this step, replacing Lemma \ref{Gengradingchange} with Lemma \ref{Gradingchange} and using the linear growth of relative Alexander gradings.\\

Combining these facts, we will have that the map $t^{m'}_{m}$ is a bijection on top $k$ Alexander gradings which also preserves relative Alexander grading based on Equation \ref{TopAlexandersaturatioon} and Lemma \ref{Gengradingchange}. As a result, we have the following isomorphism of Abelian groups: 
$$\bigoplus_{j=0,\dots,k-1} t^{m'}_{m}(j) : \bigoplus_{j=0,\dots,k-1} \widehat{CFK}(\mathcal{H}_{m},R_{m}-j) \longrightarrow \bigoplus_{j=0,\dots,k-1} \widehat{CFK}(\mathcal{H}_{m'},R_{m'}-j)$$
The same localization argument proves that $\bigoplus_{j=0,\dots,k-1} t^{m'}_{m}(j)$ is a chain map and as a result induces the following isomorphisms on knot Floer homology
$$\forall \ 0 \leq j \leq k-1\ : \ (t^{m'}_{m}(j))_{*}: \widehat{HFK}(\mathcal{H}_{m},R_{m}-j) \xrightarrow{\qquad \cong \qquad} \widehat{HFK}(\mathcal{H}_{m'},R_{m'}-j)$$
The final step is to show that the stabilization region reaches the non-trivial knot Floer homologies i.e. Alexander gradings in the interval $[-g(K_m), g(K_m)]$. This also exactly follows the computation for $n=3$. The main idea is that $g(K_m)$ grows linearly with respect to $m$ with the slope $\frac{n(n-1)}{2}$. Based on Lemma \ref{Gengradingchange}, the span of knot Floer homology of $\widehat{HFK}(K_m)$, which is equal to $R_m-L_m$, grows with slope at most $n(n-1)$. This means that $R_m$ can't grow with a slope higher than $\frac{n(n-1)}{2}$, as it means that $L_m$ decreases with a slope less than $\frac{n(n-1)}{2}$. This contradicts the fact that $-g(K_m)$ decreases with slope $\frac{n(n-1)}{2}$ and $\widehat{HFK}(K_m,-g(K_m)) \neq 0$. 
\end{proof}
We established the stabilization phenomenon and showed it reaches the non-trivial knot Floer homologies. As we mentioned, based on the works of Baker and Motegi \cite{Baker2017SeifertVS}, the genus $g(K_m)$ has linear growth with slope $\frac{n(n-1)}{2}$. This means that the upper range of the knot Floer homology also has linear growth with slope $\frac{n(n-1)}{2}$, which in turn means that the shift in the Alexander grading in the stabilization phenomenon i.e. $R_{m+1}-R_{m}$ is equal to $\frac{n(n-1)}{2}$ for high enough $m$. 
\begin{coro}
    For $m>>0$, $R_m$ is a linear function of $m$ with slope $\frac{n(n-1)}{2}$.  
\end{coro}
Unfortunately, this doesn't give us any information about the growth of $deg(\Delta(K_m))$ i.e. the parameter $r$ in Theorem \ref{Chenstablize}.\\

We end this paper with a remark about the Maslov grading. Both Lemmas \ref{Gradingchange} and \ref{Gengradingchange}, show that similar to Alexander grading, (signed) distance of the Maslov gradings of generators inside $\mathfrak{G}^{n-1}(\mathcal{H}_m)$ and other generators grows linearly under the map $t_m$ as well. The slope of growth for Maslov grading is smaller, but it does not affect the arguments leading to Equations \ref{saturation} and \ref{TopAlexandersaturatioon}. As a result, we can derive a similar equation about the Maslov grading. The following corollary formulates this result. 
\begin{coro}
Let $B_m$ be the highest Maslov grading of a generator of $\widehat{CFK}(\mathcal{H}_m)$. For any $k \in \mathbb{N}$, there exists $m_k$, such that for any $m \geq m_k$ we have the following:
$$\forall \ 0 \leq j \leq k-1 \ \  \widehat{CFK}_{B_m-j}(\mathcal{H}_m) \subset \langle \mathfrak{G}^{n-1}(\mathcal{H}_m)\rangle.$$
\end{coro}

\bibliography{bibtemplate}

\providecommand{\bysame}{\leavevmode ---\ }
\providecommand{\og}{``}
\providecommand{\fg}{''}
\providecommand{\smfandname}{et}
\providecommand{\smfedsname}{\'eds.}
\providecommand{\smfedname}{\'ed.}
\providecommand{\smfmastersthesisname}{M\'emoire}
\providecommand{\smfphdthesisname}{Th\`ese}
\begin{thebibliography}{Hed05}

\bibitem[BM17]{Baker2017SeifertVS}
{\scshape K.~L. Baker {\normalfont \smfandname} K.~Motegi} -- {\og Seifert vs. slice genera of knots in twist families and a characterization of braid axes\fg}, \emph{Proceedings of the London Mathematical Society} \textbf{119} (2017).

\bibitem[Che22]{chen2022twistings}
{\scshape D.~Chen} -- {\og Twistings and the {A}lexander polynomial\fg}, 2022.

\bibitem[Hed05]{Hedden}
{\scshape M.~Hedden} -- {\og {On knot Floer homology and cabling}\fg}, \emph{Algebraic \& Geometric Topology} \textbf{5} (2005), no.~3, p.~1197 -- 1222.

\bibitem[LC16]{LambertCole2016TwistingMA}
{\scshape P.~Lambert-Cole} -- {\og Twisting, mutation and knot {F}loer homology\fg}, \emph{Quantum Topology} (2016).

\bibitem[Ras03]{Rasmussenthesis}
{\scshape J.~A. Rasmussen} -- {\og Floer homology and knot complements\fg}, \smfphdthesisname, 2003, Copyright - Database copyright ProQuest LLC; ProQuest does not claim copyright in the individual underlying works; Last updated - 2023-03-03, p.~126.

\end{thebibliography}
\bibliographystyle{smfalpha} 

\end{document}